\documentclass[12pt,reqno]{amsart}
\usepackage{latexsym,amsmath,amssymb, amsthm, mathscinet}
\usepackage{cases, verbatim}

\usepackage{amsfonts,amsmath,amsthm, amssymb}
\usepackage{cases}
\usepackage[usenames]{color}
\usepackage{enumerate}

\theoremstyle{plain}
\newtheorem{thm}{Theorem}[section]

\newtheorem{defn}{Definition}

\newtheorem{lem}{Lemma}[section]
\newtheorem{cor}[thm]{Corollary}

\newtheorem{prop}{Proposition}[section]
\theoremstyle{remark}
\newtheorem{rem}{\bf{Remark}}
\numberwithin{equation}{section}

\newcommand{\N}{\mathbb{N}}

\newcommand{\R}{\mathbb{R}}

\newcommand{\T}{\mathbb{T}}
\newcommand{\Z}{\mathbb{Z}}

\newcommand{\Li}{L^{\infty}}
\newcommand{\Lip}{{\rm Lip\,}}

\newcommand{\al}{\alpha}

\newcommand{\ep}{\varepsilon}

\newcommand{\sig}{\sigma}

\newcommand{\ul}{\underline}

\newcommand{\Div}{{\rm div}\,}

\begin{document}
\title[Quenching under forced MCF]{Quenching for axisymmetric hypersurfaces under forced mean curvature flows}

\author{Hiroyoshi Mitake}
\address[H. Mitake]{
	Graduate School of Mathematical Sciences,
	University of Tokyo
	3-8-1 Komaba, Meguro-ku, Tokyo, 153-8914, Japan}
\email{mitake@g.ecc.u-tokyo.ac.jp}

\author{Yusuke Oka}
\address[Y. Oka]{
	Graduate School of Mathematical Sciences,
	University of Tokyo
	3-8-1 Komaba, Meguro-ku, Tokyo, 153-8914, Japan}
\email{oka@ms.u-tokyo.ac.jp}

\author{Hung V. Tran}
\address[Hung V. Tran]
{
	Department of Mathematics,
	University of Wisconsin Madison, Van Vleck hall, 480 Lincoln drive, Madison, WI 53706, USA}
\email{hung@math.wisc.edu}

\thanks{
	The work of HM was partially supported by the JSPS grants: KAKENHI \#22K03382, \#21H04431, \#20H01816, \#19H00639.
  The work of YO was partially supported by JST SPRING, Grant Number JPMJSP2108.
	The work of HT was partially supported by  NSF CAREER grant DMS-1843320 and a Vilas Faculty Early-Career Investigator Award.
}

\date{\today}

\begin{abstract}
Here, we study the motion of axisymmetric hypersurfaces $\{\Gamma_t\}_{t\ge0}$ evolved by forced mean curvature flows in the periodic setting. 
We establish conditions that quenching occurs or does not occur in terms of the initial data and forcing term.
We also study the locations where the quenching happens in some special cases.
\end{abstract}

\date{\today}

%
%
%

%

\maketitle


\section{Introduction and Main Theorems}\label{sec:Intro}

In this paper we study the motion of hypersurfaces $\{\Gamma_t\}_{t\ge0}\subset\R^{n+1}$ evolved by a forced mean curvature flow
\[
V=\kappa+f \quad\text{on} \ \Gamma_t,
\]
where $V$ and $\kappa$ denote the normal velocity and $n$ times mean curvature of the hypersurfaces, respectively,
and $f:\R^{n+1}\to\R$ is a given continuous function.

We always consider the case that $\Gamma_t$ is axisymmetric, that is, $\Gamma_t$ is described by
\[
\Gamma_t=\{(x, y)\in\R\times\R^n\mid |y|=u(x,t)\},
\]
where $u:\R\times[0,\infty)\to [0,\infty)$ is a function.
Let us first derive an equation which $u$ satisfies.
Set
\[
r:=|y| \quad \text{and}\quad
\phi(x,y,t):=u(x,t)-r.
\]
Then,
$\Gamma_t=\{(x, y)\in\R\times\R^n\mid \phi(x,y,t)=0\}$.
Note that, for $y\neq 0$,
\[
D\phi=(u_x, D_y\phi)=\left(u_x,-\frac{y}{|y|}\right).
\]
By an elementary computation, we have, for $y\neq 0$,
\[
V=\frac{\phi_t}{|D\phi|}=\frac{u_t}{((u_x)^2+1)^{1/2}}.
\]
Moreover, for $y\neq 0$,
\begin{align*}
\kappa
=&\,
\Div\left(\frac{D\phi}{|D\phi|}\right)
=
\Div\left(\frac{(u_x, D_y\phi)}{((u_x)^2+1)^{1/2}}\right)\\
=&\,
\frac{\partial}{\partial x}\left(\frac{u_x}{((u_x)^2+1)^{1/2}}\right)+\Div_y\left(\frac{D_y\phi}{((u_x)^2+1)^{1/2}}\right)
=
\frac{u_{xx}}{((u_x)^2+1)^{3/2}}+\Div_y\left(\frac{-\frac{y}{|y|}}{((u_x)^2+1)^{1/2}}\right)\\
=&\,
\frac{u_{xx}}{((u_x)^2+1)^{3/2}}-\frac{n-1}{((u_x)^2+1)^{1/2}|y|}
=\frac{u_{xx}}{((u_x)^2+1)^{3/2}}-\frac{n-1}{((u_x)^2+1)^{1/2}u}.
\end{align*}
We always suppose that $f$ is independent of the $y$-variable, that is,
\[
f(x,y)=f(x) \qquad\text{for all} \ (x,y)\in \R\times\R^n.
\]
Then,
\begin{equation}\label{eq:MC-rot}
\left\{
\begin{array}{ll}
\displaystyle
u_t=\frac{u_{xx}}{(u_x)^2+1}-\frac{n-1}{u}+f(x)\sqrt{(u_x)^2+1} \quad & \text{for} \ x\in\R, \ t>0, \\
u(\cdot,0)=g, & \text{for} \ x\in\R,
\end{array}
\right.
\end{equation}
where $g:\R\to(0,\infty)$ is a given continuous function.

\medskip
Now, our interest is to see if the quenching phenomenon occurs, that is, whether there exists $(x_{0},T)\in\R\times\left(0,\infty\right)$ such that
\[
u(x_0,T)=0
\]
or not.
We define
\begin{equation}\label{def:T}
T^\ast=T^\ast(f,g)
:=
\sup\big\{t>0\mid \inf_{x\in\R}u(x,t)>0\big\}\in(0,\infty],
\end{equation}
and we call $T^\ast$ the \textit{quenching time} for \eqref{eq:MC-rot}.
We say that quenching occurs if $T^{\ast}<\infty$, and does not occur if $T^{\ast}=\infty$.
We investigate this phenomenon in the periodic setting, and thus we \textit{always} assume throughout the paper
\begin{itemize}
\item[(A1)]
$f\in C^1(\T, (0,\infty))$, and $g\in C^2(\T, (0,\infty))$,
where we denote by $\T:=\R/\Z$ the one dimensional flat torus.
\end{itemize}
If a function $h:\R\to\R$ is $\Z$-periodic, we can think of $h$ as a function from $\T$ to $\R$ as well, and vise versa.
From now on, we switch freely between the two interpretations.
In this paper, we want to study the quenching phenomenon in terms of the competition between the forcing term $f(x)$, and the mean curvature of the initial surface, we introduce a parameter $\ep>0$ as follows:
\begin{equation}\label{eq:periodic}
\left\{
\begin{array}{ll}
\displaystyle
u^\ep_t=\frac{u^\ep_{xx}}{(u^\ep_x)^2+1}-\frac{n-1}{u^\ep}
+\frac{1}{\ep}f(x)\sqrt{(u^\ep_x)^2+1} & \text{for} \ x\in \T, \ t\in(0, T_\ep^\ast), \\
u^\ep(x,0)=\ep^\alpha g(x) & \text{for} \ x\in\T,
\end{array}
\right.
\end{equation}
where $\alpha\in(0,\infty)$ is a given constant, and $T_\ep^\ast:=\sup\left\{t>0\mid \inf_{x\in\T}u^\ep(x,t)>0\right\}\in(0, \infty]$ is the quenching time for \eqref{eq:periodic}.
Basically, we scale the original forcing term by $\ep^{-1}$ and the original initial data by $\ep^{\alpha}$ and aim at observing the quenching behavior of \eqref{eq:periodic} in light of the scaling factors. 

Naively, for small $\ep>0$, if $\alpha>1$, then  the curvature term is stronger than the forcing term, which means that we can expect the quenching happens for small $\ep>0$. 
On the other hand if $0<\alpha<1$, then the curvature term is weaker than the forcing term, and we can expect the quenching never happens. 
Indeed we are able to justify this intuition by using some ODE analysis.

\begin{prop}\label{prop:1}
Assume {\rm(A1)}.
Let $\ep, \alpha>0$, and $u^\ep$ be the solution to \eqref{eq:periodic}, and let us denote $T_\ep^\ast$ be the quenching time for \eqref{eq:periodic}.
There exists $\ep_0>0$ such that for all $\ep\in(0,\ep_0)$ if $\alpha>1$, then $T_\ep^\ast<\infty$, and if $\alpha<1$, then $T_\ep^\ast=\infty$.
\end{prop}

See Section \ref{sec:pfthm1} for the proof of Proposition \ref{prop:1}.
The idea of constructing sub/supersolutions to imply the lower/upper bound for the sets enclosed by $\Gamma_t$ is natural. See \cite{GMT, Su, An, Gr} for instance.

\medskip
In the critical case $\alpha=1$, the situation is much more subtle.
The quenching/non-quenching depends on the interactions between the initial data and forcing term.
Here is one of our main theorems.

\begin{thm}\label{thm:1}
Assume {\rm(A1)} holds.
Assume $\al=1$.
Let $u^{\ep}$ be the solution to \eqref{eq:periodic}.
\begin{enumerate}[{\rm(a)}]
\item
If $\min_{x\in\T}f(x)g(x)>n-1$, then there exist $\ep_0>0$ and $c_0>0$ such that
\[
u^\ep(x,t)\ge \ep c_0 \quad\text{for all} \ x\in\T, \ t\in[0,\infty), \ \ep\in(0,\ep_0).
\]
In short, for $\ep\in(0,\ep_0)$, the quenching does not happen, that is $T^\ast_\ep=\infty$.
\item
If $\max_{x\in\T}f(x)\max_{x\in\T}g(x)<n-1$, then for all $\ep>0$, the quenching happens, that is $T^{\ast}_{\ep}<\infty$.
\end{enumerate}
\end{thm}
Except for the cases in Theorem \ref{thm:1}  (a), (b), it is not clear yet to see whether the quenching/non-quenching happens.

We introduce a notion of the quenching in the limit of $\ep$.
\begin{defn}
  Assume $\al=1$.
  Let $u^{\ep}$ be the solution to \eqref{eq:periodic}.
If there exist $\{\ep_k\}_{k\in\N}$ with $\ep_k\to 0$ as $k\to\infty$, $(x_k, t_k)\in\T\times(0,\infty)$ such that
\[
\frac{u^{\ep_{k}}(x_k, t_k)}{\ep_k}\to 0 \quad\text{as} \ k\to\infty,
\]
then we say that an {\rm$\ep$-limit quenching} happens.
\end{defn}

\begin{thm}\label{thm:2}
Assume {\rm(A1)} holds.
Assume $\al=1$.
Let $u^{\ep}$ be the solution to \eqref{eq:periodic}.
If $\max_{x\in\T}f(x)g(x)<n-1$, then an $\ep$-limit quenching happens.
\end{thm}

To prove Theorem \ref{thm:2}, a priori estimate on the derivatives of the solution is essential,
which is established in Corollary \ref{cor:apriori} based on the Bernstein method.

We finally study the locations where the quenching happens under an additional symmetric condition:
\begin{itemize}
\item[(A2)]
$f(x)=f(1-x)$, $g(x)=g(1-x)$ for all $x\in\left[0,1\right]$, and
$f'(x)\ge0$, $g'(x)\ge0$ for all $x\in[0,\frac{1}{2}]$ with $g(0)<g(\frac{1}{2})$.
\end{itemize}
We consider \eqref{eq:periodic} only in the case $\ep=1$ for simplicity.
Then, \eqref{eq:periodic} becomes \eqref{eq:MC-rot}.
We establish the following result.
\begin{thm}\label{thm:3}
Assume {\rm(A1), (A2)} hold, and $\max_{\T}f \max_{\T}g<n-1$.
Let $u$ be the solution to \eqref{eq:MC-rot}.
Then it holds that
\begin{align*}
&\lim_{t\to (T^{\ast})^{-}} u(k,t)=0 \quad\text{for all} \ k\in\Z, \\
&\lim_{t\to (T^{\ast})^{-}} u(x,t)>0 \quad\text{for all} \ x\in\R\setminus\Z.
\end{align*}
\end{thm}

\begin{rem} 
Note that instead of the periodic setting, we can also consider \eqref{eq:periodic} in $(0,1)\times (0,\infty)$ with the homogeneous Neumann boundary condition $u^\ep_x(0,t) =u^\ep_x(1,t)=0$. All of our main results (Proposition \ref{prop:1} and Theorems \ref{thm:1}--\ref{thm:3}) hold true in this setting. We refer the reader to \cite{JKMT} for regularity and large time behavior results of level-set forced mean curvature flows with the Neumann boundary condition.
\end{rem}

\medskip
We give a non exhaustive list of results related to our paper.
The quenching phenomena for the initial-boundary value problem of semilinear as well as quasilinear parabolic equations have been well studied (see \cite{K, AW, FKL, DZ} for instance).
Geometrically if we consider the mean curvature flow, then we may have the singularity of surfaces which is sometimes called the quenching, or the pinching. 
There have been many works to study the singularity of the mean curvature flow, and 
%
%
therefore we here only focus on the mean curvature flow under the axisymmetric setting.
Under some relatively restrictive settings, the location of the quenching has been studied, and see \cite{DK, EM}.
We also refer to \cite{SS, GS} for the study of asymptotic profiles of solutions at the quenching places,
and to \cite{AAG, M} for the study of number of times that quenching occurs.
Non-compact surfaces are also studied in \cite{GSU1, GSU2}.
It is worth emphasizing that as far as the authors know there has not been any work studying the quenching phenomena for the forced mean curvature flow.

\bigskip
The paper is organized as follows.
Section \ref{sec:pfthm1} is devoted to give the proofs of Proposition \ref{prop:1} and Theorem \ref{thm:1} by constructing simple subsolutions and supersolutions.
In Section \ref{sec:ep-quenching}, we first establish a priori estimate on the derivatives of the solution to \eqref{eq:periodic} in Corollary \ref{cor:apriori}, and apply Corollary \ref{cor:apriori} to prove Theorem \ref{thm:2}. We then give the proof of Theorem \ref{thm:3} in Section \ref{sec:quen-place}.

\section{Quenching and Non-quenching}\label{sec:pfthm1}

We begin with this section by giving the proof of Proposition \ref{prop:1}.

\begin{proof}[Proof of Proposition {\rm\ref{prop:1}}]
Here, we construct a spatially constant super/subsolution to \eqref{eq:periodic}.
For $a, b>0$, let us consider the ordinary differential equation
\begin{equation}\label{eq:ODE}
\left\{
\begin{array}{ll}
\dot{y}=-\frac{n-1}{y}+\frac{a}{\ep} & \text{for} \ t>0, \\
y(0)=\ep^\alpha b.
\end{array}
\right.
\end{equation}
We see that
\begin{align}
\dot{y}(0)
=-\frac{n-1}{y(0)}+\frac{a}{\ep}
=-\frac{n-1}{\ep^\alpha b}+\frac{a}{\ep}
<0 & \quad
\text{if} \ \alpha>1, \quad\text{and}
\label{eq:ob1}\\
\dot{y}(0)
=-\frac{n-1}{y(0)}+\frac{a}{\ep}
=-\frac{n-1}{\ep^\alpha b}+\frac{a}{\ep}
>0 & \quad
\text{if} \ \alpha<1
\label{eq:ob2}
\end{align}
for sufficiently small $\ep>0$, which implies the monotonicity of $y$ in a  neighborhood of $t=0$.
Thus, if $\alpha>1$, then
\[
\dot{y}(t)=-\frac{n-1}{y(t)}+\frac{a}{\ep}
\le -\frac{n-1}{y(0)}+\frac{a}{\ep}<0
\quad\text{for all} \ t>0.
\]
Also, if $\alpha<1$, then
\[
\dot{y}(t)=-\frac{n-1}{y(t)}+\frac{a}{\ep}
\ge -\frac{n-1}{y(0)}+\frac{a}{\ep}>0
\quad\text{for all} \ t>0.
\]
Therefore, letting $\ep>0$ be sufficiently small, we obtain that $y(t)$ is strictly decreasing if $\al>1$, and $y(t)$ is strictly decreasing if $\al<1$.

In case $\alpha>1$, taking $a:=\max_{\T} f$ and $b:=\max_{\T}g$, we obtain a supersolution
$v^\ep(x,t):=y(t)$ to \eqref{eq:periodic}, which implies that there exists $\tilde{T}_\ep>0$ such that
\[
v^\ep(x,\tilde{T}_\ep)=0
\quad\text{for all} \ x\in\T.
\]
By the comparison principle, we get the inequality
\begin{align*}
  0\le u^\ep(x,t)\le v^\ep(x,t)
  \quad\text{for all} \ (x,t)\in\T\times\left(0,T_{\ep}^{\ast}\wedge\tilde{T}_{\ep}\right). 
\end{align*}
Combining the above two, $u^{\ep}$ quenches in finite time.

In case $0<\alpha<1$, taking $a:=\min_{\T} f$ and $b:=\min_{\T}g$, we obtain a subsolution
$w^\ep(x,t):=y(t)$ to \eqref{eq:periodic}, which implies that there exists $C>0$ such that
\[
C\leq w^\ep(x,t)\le u^\ep(x,t)\quad\text{for all} \ x\in\T, t\in[0,\infty),
\]
which implies $T^\ast_\ep=+\infty$.
\end{proof}

\begin{proof}[Proof of Theorem {\rm\ref{thm:1}}]
  Setting $v^\ep(x,t):=\frac{1}{\ep}u^\ep(x,t)$, we rewrite \eqref{eq:periodic} as
  \begin{equation}\label{eq:periodic-v}
  \left\{
  \begin{array}{ll}
  \displaystyle
  \ep^2\left(v^\ep_t-\frac{v^\ep_{xx}}{\ep^2(v^\ep_x)^2+1}\right)
  =
  f(x)\sqrt{\ep^2(v^\ep_x)^2+1}-\frac{n-1}{v^\ep}
  & \text{for} \ x\in \T, \ t>0, \\
  v^\ep(x,0)=g(x) & \text{for} \ x\in\T.
  \end{array}
  \right.
  \end{equation}

We first prove (a).
By the assumption, there exists $\delta>0$ such that
\[
(f(x)-\delta) g(x)>n-1 \quad\text{for} \ x\in\T.
\]
We now prove that $\psi(x):=\frac{n-1}{f(x)-\delta}$ is a subsolution to \eqref{eq:periodic-v}.
It is clear to see $\psi(x)<g(x)$ for all $x\in\T$.
Noting that
\begin{align*}
& \ep^2\left(\psi_t-\frac{\psi_{xx}}{1+\ep^2(\psi_x)^2}\right)\le \ep^2|\psi_{xx}|\le C(\delta)\ep^2, \\
& f(x)\sqrt{\ep^2(\psi_x)^2+1}-\frac{n-1}{\psi}=f(x)\big(\sqrt{\ep^2(\psi_x)^2+1}-1\big)+\delta\ge\delta
\end{align*}
for some $C(\delta)>0$.
Thus, taking $\ep_0>0$ so that $\ep_0\le \sqrt{\frac{\delta}{C(\delta)}}$, we see that $\psi$ is a subsolution to \eqref{eq:periodic-v} for all $\ep\in(0,\ep_0)$.
Therefore, by the comparison principle, we obtain
\[
\frac{u^\ep(x,t)}{\ep}=v^\ep(x,t)\ge \psi(x)\ge\min_{\T}\psi=c_0>0
\]
for all $x\in\T$, $t\in[0,\infty)$, and $\ep\in(0,\ep_0)$, which implies the conclusion.

To prove (b), setting $a:=\max_{\T}f, b:=\max_{\T}g$, we let $y$ be the solution to \eqref{eq:ODE} with $\alpha=1$.
Note that
\[
\dot{y}(0)=\frac{1}{\ep\max_{\T}g}\big(\max_{\T}f\max_{\T}g-(n-1)\big)<0,
\]
which implies that $y$ is strictly decreasing.
Therefore, by using a similar argument to that of the proof in Proposition \ref{prop:1} for $\alpha>1$, we yield that $T^\ast_\ep<+\infty$.
\end{proof}

\section{$\ep$-limit quenching}\label{sec:ep-quenching}

To prove Theorem \ref{thm:2}, we first establish a priori estimate on the derivatives of the solution.
In this section, we always assume that $\al=1$.

\begin{prop}\label{prop:est1}
Let $\tau>0$. If there exists $c_\ast>0$ such that $u^\ep(x,t)\ge \ep c_\ast$ for all
$(x,t)\in \T\times(0,\tau)$, then
\[
|\ep u_t^\ep(x,t)|\le Ce^{\frac{n-1}{\ep^{2}c_{\ast}^{2}}t} \quad \text{for all} \ (x,t)\in\T\times\left(0, \tau\right)
\]
for some $C=C(\tau)\ge0$. In particular, for $k\in\N$, there exist a constant $C=C(\tau,k)>0$ such that
\begin{align*}
  |\ep u_{t}^{\ep}(x,t)|\leq C
  \quad \text{for all} \ (x,t)\in\T\times\left(0, \tau\land\frac{k\ep^{2}c_{\ast}^{2}}{n-1}\right).
\end{align*}
\end{prop}
\begin{proof}
We define $\psi^\pm$ by
\[
\psi^{\pm}(x,t):=\ep g(x)\pm \frac{C_{1}}{\ep}t,
\]
where
\[
C_{1}:=\ep^{2}\max_{x\in\T}|g''(x)|+\overline{f}\sqrt{\ep^2\max_{\T}(g')^2+1}+\frac{2(n-1)}{\ul{g}}>0,
\]
and $\overline{f}=\max_{\T}f$, and $\underline{g}=\min_{\T}g$.

Then, $\psi^{\pm}$ is a supersolution and subsolution to \eqref{eq:periodic} for $t\in (0,\sigma_1)$,
where $\sigma_1:=\frac{\ep^2\underline{g}}{2C_1}$.
Indeed, noting that $\psi^+_t=\frac{C_1}{\ep}$, $\psi^+_x=\ep g'$, $\psi^+_{xx}=\ep g''$,
\begin{align*}
&\psi^{+}_t-\frac{\psi^{+}_{xx}}{(\psi^{+}_x)^2+1}+\frac{n-1}{\psi^{+}}
-\frac{1}{\ep}f(x)\sqrt{(\psi^{+}_x)^2+1}\\
\ge & \
\frac{C_1}{\ep}
-\left(\ep\max|g''|
+\frac{\overline{f}}{\ep}\sqrt{\ep^2\max_\T (g')^{2}+1}\right)>0.
\end{align*}
To prove that $\psi^-$ is a subsolution to \eqref{eq:periodic}, noting that
\[
\frac{n-1}{\psi^-}
=
\frac{n-1}{\ep g-\frac{C_1t}{\ep}}
\le
\frac{n-1}{\ep \underline{g}-\frac{C_1\sig_1}{\ep}}
=\frac{2(n-1)}{\ep \underline{g}}
\quad\text{for all} \quad t\in[0,\sigma_1],
\]
we can similarly conclude
\[
\psi^{-}_t-\frac{\psi^{-}_{xx}}{(\psi^{-}_x)^2+1}+\frac{n-1}{\psi^{-}}
-\frac{1}{\ep}f(x)\sqrt{(\psi^{-}_x)^2+1}
\le 0.
\]
By the comparison principle, we obtain
\begin{equation}\label{ineq:initial}
|u^{\ep}(x,t)-\ep g(x)|\le \frac{C_1t}{\ep}
\quad\text{for all} \ t\in[0,\sig_1].
\end{equation}
Especially, we have
\begin{align*}
  |u^{\ep}_{t}(x,0)|
  \leq\frac{C_{1}}{\ep}.
\end{align*}

Next, we differentiate \eqref{eq:periodic} with respect to $t$ to yield
\begin{align}\label{eq:periodic-t}
  (u^{\ep}_{t})_{t}
  =
  \frac{(u^{\ep}_{t})_{xx}}{(u^{\ep}_{x})^{2}+1}
  -\frac{2u^{\ep}_{xx}u^{\ep}_{x}(u^{\ep}_{t})_{x}}{((u^{\ep}_{x})^{2}+1)^{2}}
  +\frac{n-1}{(u^{\ep})^{2}}u^{\ep}_{t}
  +\frac{f(x)}{\ep}
  \frac{u^{\ep}_{x}(u^{\ep}_{t})_{x}}{\sqrt{(u^{\ep}_{x})^{2}+1}}.
\end{align}
Set $\phi(x,t):=u^{\ep}_{t}(x,t)$. 
Then, by \eqref{eq:periodic-t}, 
\begin{equation}\label{eq:phi}
\phi_t=
  \frac{\phi_{xx}}{(u^{\ep}_{x})^{2}+1}
  -\frac{2u^{\ep}_{xx}u^{\ep}_{x}\phi_{x}}{((u^{\ep}_{x})^{2}+1)^{2}}
  +\frac{n-1}{(u^{\ep})^{2}}\phi
  +\frac{f(x)}{\ep}\frac{u^{\ep}_{x}\phi_{x}}{\sqrt{(u^{\ep}_{x})^{2}+1}}
  \quad\text{for all}\quad
  0<t<\tau.
\end{equation}
Let $\phi_{\pm}(x,t):=\pm\frac{C_{1}}{\ep}\exp\left(\frac{n-1}{\ep^{2}c_{\ast}^{2}}\,t\right)$. 
Note that
\begin{align*}
&(\phi_{-})_t(x,t)
  =
  \frac{n-1}{\ep^{2}c_{\ast}^{2}}\phi_{-}(x,t)
  \leq
  \frac{n-1}{(u^{\ep})^{2}}\phi_{-}(x,t)
  \quad\text{as}\quad
  \phi_{-}(x,t)<0, \\
&(\phi_{+})_t(x,t)
  =
  \frac{n-1}{\ep^{2}c_{\ast}^{2}}\phi_{+}(x,t)
  \geq
  \frac{n-1}{(u^{\ep})^{2}}\phi_{+}(x,t)
  \quad\text{as}\quad
  \phi_{+}(x,t)>0. 
\end{align*}
Thus, $\phi_\pm$ are a subsolution and a supersolution to \eqref{eq:phi}, respectively. 
We use the comparison principle to get
\[
\phi_{-}(x,t)
\le\phi(x,t)=u^{\ep}(x,t)
\le\phi_{+}(x,t)
\quad
\text{for all} \ 
(x,t)\in\T\times(0,\tau), 
\]
which implies the conclusion. 
\qedhere
\end{proof}

\begin{prop}\label{prop:est2}
Let $\tau>0$. If there exists $c_\ast>0$ such that $u^\ep(x,t)\ge \ep c_\ast$ for all
$(x,t)\in \T\times(0,\tau)$, then for all $k\in\N$,
\[
|u_x^\ep(x,t)|\le C \quad \text{for all} \ (x,t)\in\T\times\left(0, \tau\land\frac{k\ep^2 c_\ast^2}{n-1}\right)
\]
for some $C=C(\tau,k)>0$.
\end{prop}
\begin{proof}
Differentiate \eqref{eq:periodic} with respect to $x$ to obtain
\[
u^{\ep}_{xt}=
\frac{u^{\ep}_{xxx}}{(u^{\ep}_{x})^{2}+1}
-
\frac{2u^{\ep}_{x}(u^{\ep}_{xx})^{2}}{((u^{\ep}_{x})^{2}+1)^{2}}
+\frac{n-1}{(u^{\ep})^{2}}u^{\ep}_{x}
+\frac{f'(x)}{\ep}\sqrt{(u^{\ep}_{x})^{2}+1}
+\frac{f(x)u^{\ep}_{x}u^{\ep}_{xx}}{\ep\sqrt{(u^{\ep}_{x})^{2}+1}}.
\]
Multiply by $u^{\ep}_{x}$ and set $\psi:=\frac{1}{2}(u^{\ep}_{x})^{2}$ to obtain
\[
\psi_{t}=
\frac{\psi_{xx}-(u^{\ep}_{xx})^{2}}{(u^{\ep}_{x})^{2}+1}
-\frac{2\psi_{x}^{2}}{((u^{\ep}_{x})^{2}+1)^{2}}
+\frac{(n-1)(u^{\ep}_{x})^{2}}{(u^{\ep})^{2}}
+\frac{f'(x)}{\ep}u^{\ep}_{x}\sqrt{(u^{\ep}_{x})^{2}+1}
+\frac{f(x)u^{\ep}_{x}\psi_{x}}{\ep\sqrt{(u^{\ep}_{x})^{2}+1}}.
\]

Take $(x_0, t_0)\in\T\times[0,\tau\land\frac{k\ep^2c_\ast^2}{n-1}]$ so that
$\psi(x_0,t_0)=\max_{\T\times[0,\tau\land\frac{k\ep^2c_\ast^2}{n-1}]}\psi$.
If $t_0=0$, then there is nothing to prove. Thus, assume $t_0>0$.
Then,
$\psi_{x}(x_0,t_0)=0$, $\psi_{t}(x_0,t_0)\ge 0,\psi_{xx}(x_0,t_0)\le 0$.
Therefore, at $(x,t)=(x_0, t_0)$,
\begin{align}
  \frac{(u^{\ep}_{xx})^{2}}{(u^{\ep}_{x})^{2}+1}
  &\leq
  \frac{(n-1)(u^{\ep}_x)^2}{(u^{\ep})^2}
+\frac{f'(x)}{\ep}u^{\ep}_{x}\sqrt{(u_{x}^{\ep})^{2}+1}
\le
  \frac{(n-1)(u^{\ep}_x)^2}{\ep^2c_\ast^2}
+\frac{f'(x)}{\ep}u^{\ep}_{x}\sqrt{(u^{\ep}_{x})^{2}+1}
\nonumber \\
  &\leq
  \frac{(n-1)(u^{\ep}_x)^2}{\ep^2c_\ast^2}+\frac{C}{\ep}((u^{\ep}_x)^2+1)
  \le \frac{C}{\ep^2}(u^{\ep}_x)^2+\frac{C}{\ep}
\label{ineq:Bern1}
\end{align}
for some $C\ge0$, which is independent of $\ep$.

Moreover,
\[
\frac{(u^{\ep}_{xx})^{2}}{(u^{\ep}_{x})^{2}+1}
=
\left((u^{\ep}_{x})^{2}+1\right)\left(\frac{u^{\ep}_{xx}}{(u^{\ep}_{x})^{2}+1}\right)^{2}
=\left((u^{\ep}_{x})^{2}+1\right)
  \left(
  u^{\ep}_{t}+\frac{n-1}{u^{\ep}}-\frac{1}{\ep}f(x)\sqrt{(u^{\ep}_{x})^{2}+1}
  \right)^{2}.
\]
Setting
$X:=\frac{n-1}{u^{\ep}}-\frac{1}{\ep}f(x)\sqrt{(u^{\ep}_{x})^{2}+1}$, we obtain
\begin{equation}\label{ineq:Bern2}
\frac{(u^{\ep}_{xx})^{2}}{(u^{\ep}_{x})^{2}+1}
\ge
\left((u^{\ep}_{x})^{2}+1\right)\left(\frac{X^2}{2}-(u^{\ep}_t)^2\right).
\end{equation}
Here, by the Young inequality, for any $\rho\in(0,1)$,
\begin{align}
X^2
&=
  \frac{(n-1)^{2}}{(u^{\ep})^{2}}
  -\frac{2(n-1)}{\ep u^{\ep}}f(x)\sqrt{(u^{\ep}_{x})^{2}+1}
  +\frac{1}{\ep^{2}}f(x)^{2}\left((u^{\ep}_{x})^{2}+1\right)
  \nonumber \\
 &\ge
  -\frac{2(n-1)}{\ep u^{\ep}}\overline{f}\sqrt{(u^{\ep}_{x})^{2}+1}
  +\frac{1}{\ep^{2}}\underline{f}^{2}\left((u^{\ep}_{x})^{2}+1\right)
  \nonumber \\
&\geq
-\frac{(n-1)}{\ep^{2}c_\ast}\left(\frac{\overline{f}^{2}}{\rho}+\rho((u^{\ep}_{x})^{2}+1)\right)
  +\frac{1}{\ep^{2}}\underline{f}^{2}(u^{\ep}_{x})^{2}
  \notag
  \\
  &=
  \frac{1}{\ep^{2}}\bigg(
  \underline{f}^{2}
  -\frac{(n-1)}{c_\ast}\rho
  \bigg)(u^{\ep}_{x})^{2}
  -
  \frac{(n-1)}{\ep^{2}c_\ast}\left(\frac{\overline{f}^{2}}{\rho}+\rho\right)
  =:   \frac{1}{\ep^{2}}\big(C^1_\rho (u^{\ep}_x)^2-C_\rho^2\big),
\label{ineq:Bern3}
\end{align}
where we set $\overline{f}:=\max_{\T}f$, $\underline{f}:=\min_{\T}f$.

Take $\rho\in(0,1)$ so small that $C^{1}_{\rho}>0$.
Combining \eqref{ineq:Bern1} and \eqref{ineq:Bern2} with \eqref{ineq:Bern3}, we obtain
\begin{align*}
\frac{C}{\ep^2}(u^{\ep}_x)^2+\frac{C}{\ep}
&\ge
\left((u^{\ep}_{x})^{2}+1\right)\left(\frac{X^2}{2}-(u^{\ep}_t)^2\right)\\
&\ge
\frac{1}{2}\left((u^{\ep}_{x})^{2}+1\right)
\cdot\frac{1}{\ep^2}\big(C^1_\rho (u^{\ep}_x)^2-C_\rho^2\big)
-
\left((u^{\ep}_{x})^{2}+1\right)\cdot\left(\frac{C}{\ep}\right)^2
\end{align*}
by using Proposition \ref{prop:est1}.
Summing up, we obtain
\[
(u^{\ep}_x)^4\le C (u^{\ep}_x)^2+C
\]
for some $C\ge0$ which is independent of $\ep>0$. Therefore, $|u^{\ep}_x|\le C$ for some $C\ge0$.
\end{proof}

The following corollary is a direct consequence of Propositions \ref{prop:est1}, \ref{prop:est2}.
\begin{cor}\label{cor:apriori}
If there exists $c_\ast>0$ such that
\begin{equation}\label{cond:no-quench}
u^\ep(x,t)\ge \ep c_\ast
\quad\text{for all} \
(x,t)\in \T\times[0,\infty),
\end{equation}
then for all $k\in\N$,
\[
|\ep u_t^\ep(x,t)|+|u_x^\ep(x,t)|\le C \quad \text{for all} \ (x,t)\in\T\times\left[0,\frac{k\ep^{2}c_{\ast}^{2}}{n-1}\right)
\]
for some $C=C(k)\geq 0$ independent of $\ep\in(0,1)$.
\end{cor}

We note that Lemma \ref{lem:ODE-vw} is a special case of a comparison principle for first-order Hamilton--Jacobi equations of the type $u_{t}+H(x,u,Du)=0$, where $H$ is Lipschitz continuous in the $u$-variable.
See \cite{Hung-book} for example.

\begin{lem}\label{lem:ODE-vw}
Let $T>0$, and $v, w\in C([0,T])$ be a viscosity subsolution, and supersolution to
\begin{equation}\label{eq:ODE-vw}
u_t(t)=-\frac{n-1}{u(t)}+f\quad\text{for} \ t\in(0,T],
\end{equation}
where $f$ is a positive constant, respectively.
Assume that $v(t), w(t)\ge c_\ast$ for all $t\in[0,T]$ and some $c_\ast>0$.
If $v(0)\le w(0)$, then $v(t)\le w(t)$ for all $t\in[0,T]$.
\end{lem}

\begin{proof}
Set $F(v):=-\dfrac{n-1}{v}+f$.
For $K=4(n-1)/c_{\ast}^{2}$,
\begin{align*}
|F(v)-F(w)|\leq K|v-w|\quad\text{for all } \ v,w\geq\frac{c_{\ast}}{2}.
\end{align*}
Pick $a_{0}\in\left(0,1\right)$ such that
$c_{\ast}-a_{0}e^{2KT}\geq c_{\ast}/2$.
Fix $a\in\left(0,a_{0}\right]$ and set
$u(t):=v(t)-ae^{2Kt}$ for $t\in[0,T]$.
Then we have
\begin{align*}
&  u(t)\geq c_{\ast}-ae^{2Kt}\geq\frac{c_{\ast}}{2}
  \quad\text{for all }
  t\in [0,T],
  \\
&u(0)=v(0)-a<w(0).
\end{align*}
We prove that $u(t)<w(t)$ for all $t\in[0,T]$ by contradiction.
Thus, suppose $\max_{t\in[0,T]}(u-w)(t)\geq 0$.
Then, we have
\begin{align*}
  t_{0}:=\min\left\{t\in(0,T)\mid u(t)-w(t)\geq 0\right\}\in(0,T].
\end{align*}
Define an auxiliary function $\Phi:\left[0,t_{0}\right]^{2}\to\R$ by
\begin{align*}
  \Phi(t,s):=u(t)-w(s)-\frac{(t-s)^{2}}{2\ep}-(t-t_{0})^{2}.
\end{align*}
Take $(t_{\ep},s_{\ep})\in\left[0,t_{0}\right]^{2}$ so that $\Phi(t_{\ep},s_{\ep})=\max_{\left[0,t_{0}\right]^{2}}\Phi$.
Then, as usual, we have $t_{\ep}, s_{\ep}\to t_{0}$ as $\ep\to 0$.
%

Note that $v, w$ are a viscosity sub/supersolution to \eqref{eq:ODE-vw} in $(0,t_0)$. Then, we can extend the property of viscosity sub/supersolutions up to the terminal time $t=t_0$.
Therefore, by the definition of viscosity solutions, we have
\begin{align*}
  2aKe^{2Kt_{\ep}}+\frac{t_{\ep}-s_{\ep}}{\ep}+2(t_{\ep}-t_{0})-F(v(t_{\ep}))
  \leq 0
  \leq\frac{t_{\ep}-s_{\ep}}{\ep}-F(w(s_\ep)).
\end{align*}
Noting that $u(s_\ep)\le w(s_\ep)$, we have $F(u(s_\ep))\le F(w(s_\ep))$.
Thus,
\begin{align*}
  2aKe^{2Kt_{\ep}}+2(t_{\ep}-t_{0})
  &\leq
  F(v(t_{\ep}))-F(w(s_\ep))
  \\
  &=
  F(v(t_{\ep}))-F(u(s_\ep))
  +F(u(s_\ep))-F(w(s_\ep))
  \\
  &\leq
  K|v(t_{\ep})-u(s_{\ep})|.
\end{align*}
Sending $\ep\to 0$ yields
\begin{align*}
  2aKe^{2Kt_{0}}
  \leq
  K|v(t_{0})-u(t_{0})|
=
  aKe^{2Kt_{0}},
\end{align*}
which is a contradiction.
%
\end{proof}

We are now ready to prove Theorem \ref{thm:2}.
\begin{proof}[Proof of Theorem {\rm\ref{thm:2}}]
We prove by contradiction, and therefore suppose that an $\ep$-limit quenching does not happen. Then,
there exists $c_\ast>0$ such that
\[
\frac{u^\ep(x,t)}{\ep}\ge c_\ast
\qquad\text{for all} \ (x,t)\in\T\times(0,\infty), \ \text{and} \ \ep\in(0,1).
\]
By a similar argument to that in the proof of Theorem \ref{thm:1} (a) we can prove that there exists $\delta>0$ such that
$\frac{(n-1)\ep}{f(x)+\delta}$ is a supersolution to \eqref{eq:periodic}, which implies by the comparison principle that
\[
0<\ep c_{\ast}\le u^\ep(x,t)\le \frac{(n-1)\ep}{f(x)+\delta} \le C\ep
\quad\text{for all} \ (x,t)\in\T\times[0,\infty).
\]

Set $w^\ep(x,t):=\frac{u^\ep(x,\ep^2 t)}{\ep}$.
Then,
\[
\left\{
\begin{array}{ll}
\displaystyle
w^\ep_t=\frac{\ep^2w^\ep_{xx}}{\ep^2(w^\ep_x)^2+1}-\frac{n-1}{w^\ep}
+
f(x)\sqrt{\ep^2(w^\ep_x)^2+1} & \text{for} \ x\in \T, \ t>0, \\
w^\ep(x,0)=g(x) & \text{for} \ x\in\T,
\end{array}
\right.
\]
and for each $k\in \N$,
\begin{equation}\label{est:w}
\|w^\ep_{t}\|_{L^\infty(\T\times\left[0,{kc_{\ast}^{2}}/{(n-1)}\right))}=\|\ep u^\ep_t\|_{L^\infty(\T\times\left[0,{k\ep^2 c_{\ast}^{2}}/{(n-1)}\right))}\le C
\end{equation}
for some $C=C(k)>0$ which is independent of $\ep$ in view of Corollary \ref{cor:apriori}.
Define the half-relaxed limits of $w^\ep$ by
\begin{align*}
&
w^\ast(x,t):=\limsup{}^{\ast}_{\ep\to0}\{w^\rho(y,s)\mid |x-y|+|t-s|\le\rho, 0<\rho<\ep\}, \\
&
w_\ast(x,t):=\liminf{}^{\ast}_{\ep\to0}\{w^\rho(y,s)\mid |x-y|+|t-s|\le\rho, 0<\rho<\ep\}.
\end{align*}
In light of \eqref{est:w} we have $w^\ast(x,\cdot), w_\ast(x,\cdot)\in\Lip_{\text{loc}}([0,\infty))$ for all $x\in\T$.
By the stability of viscosity solutions, $w^\ast$ and $w_\ast$ are a viscosity subsolution and supersolution
to
\begin{equation}\label{eq:ODE-w}
\left\{
\begin{array}{ll}
w_t(x,t)=-\frac{n-1}{w(x,t)}+f(x) & \text{for} \ t>0,\\
w(x,0)=g(x) & \text{for all} \  x\in\T.
\end{array}
\right.
\end{equation}

By Lemma \ref{lem:ODE-vw} we obtain $w^\ast(x,\cdot)\le w_\ast(x,\cdot)$ on $[0,\infty)$ for all $x\in\T$, which implies
that $w^\ast=w_\ast$ on $\T\times[0,\infty)$. Therefore,
$w^\ep\to w$ locally uniformly on $\T\times\left[0,\infty\right)$ as $\ep\to0$.
By assumption $\max_{\T}(fg)<n-1$, there exists $\overline{x}\in\T$ such that $\beta:=f(\overline{x})-\frac{n-1}{g(\overline{x})}<0$.
Then it is clear to see that $z(t):=g(\overline{x})+\beta t$ is a supersolution to \eqref{eq:ODE-w} with $x=\overline{x}$. Since $\beta<0$, there exists $T>0$ such that $z(T)=0$, which is a contradiction.
\end{proof}

\section{The locations of quenching}\label{sec:quen-place}
In this section, we consider \eqref{eq:periodic} with $\ep=1$, that is,
\begin{equation}\label{eq:ep=1}
\left\{
\begin{array}{ll}
\displaystyle
u_t=\frac{u_{xx}}{u_x^2+1}-\frac{n-1}{u}
+f(x)\sqrt{u_x^2+1} & \text{for} \ x\in \T, \ t\in(0, T^\ast), \\
u(x,0)=g(x) & \text{for} \ x\in\T.
\end{array}
\right.
\end{equation}
Again, \eqref{eq:ep=1} here is exactly \eqref{eq:MC-rot} under the periodic setting.
The assumptions of Theorem \ref{thm:3} are always in force in this section. 
By Theorem \ref{thm:1}  (b), quenching happens and $T^*<\infty$.

\begin{lem}\label{lem:v-positive1}
Assume {\rm(A1), (A2)} hold.
We have $u_x(0,t)=u_x(\frac{1}{2},t)=0$ for all $t\in[0, T^\ast)$, and
\[
u_x(x,t)>0 \quad\text{for all} \ x\in \left(0,\frac{1}{2}\right)\times(0, T^\ast).
\]
\end{lem}
\begin{proof}
Differentiate \eqref{eq:ep=1} with respect to $x$ to yield
\[
u_{tx}=\frac{u_{xxx}}{u_x^2+1}-\frac{2u_{xx}^2u_x}{(u_x^2+1)^2}+\frac{(n-1)u_x}{u^2}+f'(x)\sqrt{u_x^2+1}+\frac{f(x)u_xu_{xx}}{\sqrt{u_x^2+1}}.
\]
Let $v=u_x$. Noting that $f'(x)\ge0$ on $[0,\frac{1}{2}]$, we have
\begin{equation}\label{ineq:u_x}
v_t\ge\frac{v_{xx}}{u_x^2+1}+\left(\frac{f(x)u_{x}}{\sqrt{u_x^2+1}}-\frac{2u_{xx}u_x}{(u_x^2+1)^2}\right)v_x+\frac{(n-1)v}{u^2}
\quad
\text{in} \ \left(0,\frac{1}{2}\right)\times(0,T^\ast).
\end{equation}
As $f(x)=f(1-x)$ and $g(x)=g(1-x)$, we see that $U(x,t):=u(1-x,t)$ is also a solution to \eqref{eq:ep=1}.
Thus, the uniqueness of solutions to \eqref{eq:ep=1} yields
\begin{equation}\label{eq:periodic-half}
u(1-x,t)=u(x,t) \quad\text{for all} \ (x,t)\in\left[0,1\right]\times[0,T^{\ast}),
\end{equation}
which, together with the fact that $u$ is $\Z$-periodic and smooth in $\R\times(0,T^\ast)$, implies
$v(0,t)=u_x(0,t)=0$ and $v(\frac{1}{2},t)=u_x(\frac{1}{2},t)=0$ for all $t \in (0,T^{\ast})$.
Thus, $v$ satisfies \eqref{ineq:u_x} and
\begin{align}
&v(0,t)=v\left(\frac{1}{2},t\right)=0
\quad
\text{for all} \ t \in (0,T^{\ast}),
\nonumber\\
&
v(x,0)=u_x(x,0)=g'(x)\ge0
\quad
\text{for all} \ x\in\left[0,\frac{1}{2}\right].
\label{cond:b-i}
\end{align}

Clearly, $\tilde{v}\equiv 0$ is a subsolution to \eqref{ineq:u_x} and \eqref{cond:b-i}.
By the comparison principle, we obtain
$u_x(x,t)=v(x,t)\ge0$ for all $(x,t)\in (0,\frac{1}{2})\times(0,T^\ast)$.
Therefore, $v$ satisfies
\[
v_t\ge\frac{v_{xx}}{u_x^2+1}+\left(\frac{f(x)u_{x}}{\sqrt{u_x^2+1}}-\frac{2u_{xx}u_x}{(u_x^2+1)^2}\right)v_x\quad
\text{in} \ \left(0,\frac{1}{2}\right)\times(0,T^\ast),
\]
which implies that, in light of the strong maximum principle,
$v>0$ in $(0,\frac{1}{2})\times(0,T^\ast)$.
\end{proof}

\medskip
We next note that $v=u_x$ also satisfies
\begin{numcases}
\displaystyle
v_t\ge \frac{v_{xx}}{(u_x)^2+1}
+\frac{f(x)vv_x}{\sqrt{(u_x)^2+1}}
-\frac{2v(v_x)^2}{((u_x)^2+1)^2}
 &
for \ $(x,t)\in (0,\frac{1}{2})\times(0, T^\ast)$, \label{v-1} \\
v(0,t)=v\left(\frac{1}{2},t\right)=0 & for \ $t\in[0,T^\ast)$, \nonumber\\
v(x,0)=g'(x)\ge0 & for $x\in[0,\frac{1}{2}]$. \nonumber
\end{numcases}

\begin{lem}\label{lem:v-positive}
Let $u$ be the solution to \eqref{eq:ep=1}, and set $v:=u_x$.
Let $a, b\in[0,\frac{1}{2}]$ with $a<b$, and $w\in C^2\big([a,b]\times[0,T^\ast]\big)$ be a subsolution to \eqref{v-1} in $(a,b)\times(0,T^\ast)$
with $w(a,t)=w(b,t)=0$, and $w(x,0)\le g'(x)$ for all $x\in[a,b]$. Then, $v\ge w$ on $[a,b]\times[0,T^\ast)$.
\end{lem}
\begin{proof}
Fix $T\in(0,T^\ast)$.
Set $\tilde{v}(x,t):=e^{-Kt}v(x,t)$, $\tilde{w}(x,t):=e^{-Kt}w(x,t)$ for all $x\in[a,b]$, $t\in[0,T]$, and $K>0$ which will be fixed later.
Then, $v$ and $w$ satisfy
\begin{align*}
&\tilde{v}_t+K\tilde{v}\ge \frac{\tilde{v}_{xx}}{(u_x)^2+1}
+\frac{e^{Kt}f(x)\tilde{v}\tilde{v}_x}{\sqrt{(u_x)^2+1}}
-\frac{2e^{2Kt}\tilde{v}(\tilde{v}_x)^2}{((u_x)^2+1)^2}
\quad\text{in} \ (a,b)\times(0, T^\ast), \\
&\tilde{w}_t+K\tilde{w}\le \frac{\tilde{w}_{xx}}{(u_x)^2+1}
+\frac{e^{Kt}f(x)\tilde{w}\tilde{w}_x}{\sqrt{(u_x)^2+1}}
-\frac{2e^{2Kt}\tilde{w}(\tilde{w}_x)^2}{((u_x)^2+1)^2}
\quad\text{in} \ (a,b)\times(0, T^\ast).
\end{align*}

Assume by contradiction that there exists $(x_0,t_0)\in (a,b)\times(0,T]$ such that
\[
\max_{[a,b]\times[0,T]}(\tilde{w}-\tilde{v})=(\tilde{w}-\tilde{v})(x_0, t_0)>0.
\]
Then,
$(\tilde{w}-\tilde{v})_t(x_0,t_0)\ge0$,
$(\tilde{w}-\tilde{v})_{x}(x_0,t_0)=0$,
$(\tilde{w}-\tilde{v})_{xx}(x_0,t_0)\le0$, which implies that
\[
K(\tilde{w}-\tilde{v})(x_0,t_0)
-\frac{e^{Kt_0}f(x_0)\tilde{w}_x(x_0,t_0)}{\sqrt{u_x^2(x_0,t_0)+1}}(\tilde{w}-\tilde{v})(x_0,t_0)
+\frac{2e^{2Kt_0}\tilde{w}_{x}^{2}(x_0,t_0)}{u_x^2(x_0,t_0)+1}(\tilde{w}-\tilde{v})(x_0,t_0)
\leq 0.
\]
Thus,
\[
K\le
\frac{e^{Kt_0}f(x_0)|\tilde{w}_x(x_0,t_0)|}{\sqrt{u_x^2(x_0,t_0)+1}}
=\frac{f(x_0)|w_x(x_0,t_0)|}{\sqrt{u_x^2(x_0,t_0)+1}}\le f(x_0)|w_x(x_0,t_0)|.
\]
Taking $K>\max_{[a,b]\times[0,T]}f(x)|w_x(x,t)|$,
we have a contradiction. Therefore, $\tilde{w}\le\tilde{v}$ on $[a,b]\times[0,T]$ for all $T\in(0,T^\ast)$,
which implies the conclusion.
\end{proof}

\begin{lem}\label{lem:w}
Let $c_0>0$, $a, b\in[0,\frac{1}{2}]$ with $a<b$, and set $w(x,t):=c_0e^{-Mt}\sin\left(\frac{\pi(x-a)}{b-a}\right)$ for $(x,t)\in[a,b]\times[0,T^\ast)$.
There exists $M>0$ such that $w$ is a subsolution to \eqref{v-1} with $w(a,t)=w(b,t)=0$ for all $t\in[0,T^\ast)$.
\end{lem}
\begin{proof}
Plug $w$ into the equation to get
\begin{align*}
&w_t-\frac{w_{xx}}{(u_x)^2+1}
+\frac{f(x)ww_x}{\sqrt{(u_x)^2+1}}
-\frac{2w(w_x)^2}{((u_x)^2+1)^2}\\
=&\,
c_0e^{-Mt}\sin\left(\frac{\pi(x-a)}{b-a}\right)
\Big[
-M+\frac{\pi^2}{(b-a)^2((u_x)^2+1)}
+\frac{f(x)\pi c_{0}e^{-Mt}}{(b-a)\sqrt{(u_x)^2+1}}\cos\left(\frac{\pi(x-a)}{b-a}\right)\\
&\hspace*{7cm}
-\frac{2\pi^{2}c_{0}^{2}e^{-2Mt}}{(b-a)^2((u_x)^2+1)^2}\cos^2\left(\frac{\pi(x-a)}{b-a}\right)
\Big]\\
\le&\,
c_0e^{-Mt}\sin\left(\frac{\pi(x-a)}{b-a}\right)\left(-M+\frac{\pi^2}{(b-a)^2}+\frac{f(x)\pi c_{0}}{b-a}\right).
\end{align*}
Thus, by choosing $M>\frac{\pi^2}{(b-a)^2}+\frac{\pi c_{0}\|f\|_{\Li(\T)}}{b-a}$, we get that $w$ is a subsolution to \eqref{v-1}
satisfying $w(a,t)=w(b,t)=0$ for all $t\in[0,T^\ast)$.
\end{proof}

\begin{proof}[Proof of Theorem {\rm\ref{thm:3}}]
By Theorem \ref{thm:1} (b) and Lemma \ref{lem:v-positive1}, it is clear to see $\lim_{t\to (T^{\ast})^{-}}u(0,t)=0$.
By \eqref{eq:periodic-half}, we only need to prove $\lim_{t\to (T^{\ast})^{-}}u(x,t)>0$ for all $x\in(0,\frac{1}{2}]$.

Let $t_0:=\frac{T^\ast}{2}$.
By Lemma \ref{lem:v-positive1}, we obtain that $v(x,t_0)=u_x(x,t_0)>0$ for all $x\in(0,\frac{1}{2})$.
Take any $0<a<a^{\ast}<b^{\ast}<b<\frac{1}{2}$. There exists $c_0>0$ such that
\[
v(x,t_0)=u_x(x,t_0)\ge c_0\ge c_0\sin\left(\frac{\pi(x-a)}{b-a}\right)
\quad\text{for all} \ x\in[a,b].
\]
In light of Lemma \ref{lem:w}, for all $t\in[t_0,T^\ast)$,
\[
u_x(x,t)\ge c_0e^{-M(t-t_0)}\sin\left(\frac{\pi(x-a)}{b-a}\right)
\quad\text{for all} \ x\in[a,b].
\]
Thus, for all $x\in[a^{\ast},b^{\ast}]$, we have
\[
u(x,t)-u(a,t)
\ge c_0e^{-\frac{MT^\ast}{2}}
\int_{a}^{a^{\ast}}\sin\left(\frac{\pi(y-a)}{b-a}\right)\, dy
=:\alpha>0,
\]
which implies
\[
\liminf_{t\to (T^{\ast})^{-}}u(x,t)\ge \limsup_{t\to (T^{\ast})^{-}}u(a,t)+\alpha\ge\alpha>0
\quad\text{for all} \ x\in[a^{\ast},b^{\ast}].
\]
We use Corollary \ref{cor:apriori} to get $u\in\Lip([a^{\ast},b^{\ast}]\times[0,T^\ast])$, which implies $\lim_{t\to (T^{\ast})^{-}}u(x,t)$ exists for all $x\in[a^{\ast},b^{\ast}]$.
By the arbitrariness of $a,b,a^{\ast},b^{\ast}$, we see that $\lim_{t\to (T^{\ast})^{-}}u(x,t)>0$ for all $x\in (0,\frac{1}{2}]$.
\end{proof}




\bibliographystyle{abbrv}


\end{document}